\documentclass[11pt,a4paper]{amsart}
\usepackage[utf8]{inputenc}
\usepackage{amsmath}
\usepackage{amsthm}
\usepackage{import}
\usepackage{geometry}
\usepackage{mathrsfs} 
\usepackage{amsfonts}
\usepackage{tkz-graph}
\usepackage{amssymb}
\usepackage{palatino}
\usepackage{hyperref}
\usepackage{tikz}
\usepackage{float}
\usetikzlibrary{matrix}
\usepackage{comment,accents}
\usepackage{enumerate,graphicx}
\usepackage{tikz}
\usetikzlibrary{arrows,shapes,positioning}
\usetikzlibrary{decorations.markings}
\usetikzlibrary{calc}

\newtheorem*{mydefn}{Definition}

\newtheorem{prop}{Proposition}[section]
\newtheorem{theorem}{Theorem}[section]
\newtheorem{lemma}{Lemma}[section]

\newtheorem{cor}{Corollary}[section]
\theoremstyle{definition}
\newtheorem{remark}{Remark}[section]
\newtheorem{myxmpl}{Example}[section]

\newcommand{\R}{\mathbb{R}}

\newcommand{\RD}{\mathbb{R}^d}

\newcommand{\cF}{\mathcal{F}}

\newcommand{\cP}{\mathcal{P}}

\newcommand{\tx}{\tilde{x}}
\newcommand{\ty}{\tilde{y}}
\newcommand{\tomega}{\tilde{\omega}}
\newcommand{\scrA}{\mathscr{A}}
\newcommand{\scrU}{\mathscr{U}}
\newcommand{\scrV}{\mathscr{V}}
\newcommand{\scrL}{\mathscr{L}}

\newcommand{\bes}{\begin{equation*}}
\newcommand{\ees}{\end{equation*}}
\newcommand{\beas}{\begin{eqnarray*}}
\newcommand{\eeas}{\end{eqnarray*}}
\newcommand{\bea}{\begin{eqnarray}}
\newcommand{\eea}{\end{eqnarray}}
\newcommand{\be}{\begin{equation}}
\newcommand{\ee}{\end{equation}}

\newcommand{\bbl}{\begin{block}}
\newcommand{\ebl}{\end{block}}

\newcommand{\LBR}{P^{x,y}_{-T,T}}
\newcommand{\LBBR}{W^{x,y}_{-T,T}}

\newcommand{\PMU}{P^{\mu}_{-T,T}}

\newcommand{\OMB}{\Omega^{x,y}_{-T,T}}
    
\begin{document}

\author{Giovanni Conforti }

\address{Universit\"at Leipzig, Fakult\"at f\"ur Mathematik und Informatik, Augustusplatz 10, 04109 Leipzig, Germany}
\email{giovanniconfort@gmail.com}

\author{Max Von Renesse}
\address{Universit\"at Leipzig, Fakult\"at f\"ur Mathematik und Informatik, Augustusplatz 10, 04109 Leipzig, Germany}
\email{renesse@math.uni-leipzig.de}

\keywords{bridges, reciprocal characteristic,coupling, gradient estimates, Logarithmic Sobolev inequality, comparison principle, Stein method.}
\date{today}
\subjclass[2010]{28D20,47D07,47D08,47A63,60J60}


\title[Quantitative results for bridges]{Couplings, gradient estimates and Logarithmic Sobolev inequalitiy for Langevin bridges}
\maketitle
\begin{abstract}
In this paper we establish quantitative results about the bridges of the Langevin dynamics and the associated reciprocal processes. They include an equivalence between gradient estimates for bridge semigroups and couplings, comparison principles,
bounds of the distance between bridges of different Langevin dynamics, and a Logarithmic Sobolev inequality for bridge measures. The existence of an invariant measure for the bridges is also discussed and quantitative bounds for the convergence to the invariant measure are proved .All results are based on a seemingly new expression of the drift of a bridge in terms of the \textit{reciprocal characteristic}, which, roughly speaking, quantifies the `` mean acceleration" of a bridge.
\end{abstract}

\section{Introduction and statement of the main results}

Bridges of Markov processes are among the basic objects of probability theory: as such, they have been intensively studied. However, quite surprisingly, there is a lack of quantitative results outside the short time regime. The aim of this article is to address this issue in the case when the underlying Markov process is a Langevin dynamics.
The Langevin dynamics $P^{x}_{-T}$ over the time interval $[-T,T]$ is the law of:
\be\label{eq:LD}
dX_{t} = -\nabla U(t,X_t)dt + dB_t, \quad X_{-T} =x \tag{\textbf{LD}(U)}
\ee
The $x y$ bridge $\LBR$ is simply obtained by conditioning $P^x_{-T}$ to the event $\{X_{T}=y\}$.
Our starting point is the observation due to Krener \cite{Kre88} (see also \cite{Blee},\cite{Cl91} ) that the family of bridges $\{P^{x,y}_{-T,T}\}_{x,y \in \RD}$ is uniquely determined by the \textit{reciprocal charachteristic}, which is a vector field derived from $U$. Its $i$-th component is the map $(t,z) \mapsto -[\scrL + \partial_t] (\partial_{z_i}U)(t,z) $, where $\scrL$ is the generator of the Langevin dynamics:
 \be\label{eq:LDgen} \scrL (f) = -\nabla U \cdot \nabla f + \frac{1}{2} \Delta f . \ee 
Another definition, which turns out to be very useful, is to consider the potential $\scrU$ defined by
\be\label{def:reciprocalpot}
\scrU(t,z) := \frac{1}{2} | \nabla U |^2(t,z) - \frac{1}{2} \Delta U(t,z) - \partial_t U(t,z).  
\ee
The reciprocal characteristic is then the field $\nabla \scrU$. To see how it arises naturally in the study of bridges consider the form of the log-density of $P^x_{-T}$ against the Wiener measure:
\[\exp \left(- \int_{-T}^T \nabla U(t,X_t) \cdot d X_t - \frac{1}{2} \int_{-T}^{T} |\nabla U(t,X_t)|^2dt \right) \]
An application of It\^o formula yields the equivalent form
\beas \exp \left(-U(1,X_1)+U(0,X_0) - \int_{-T}^{T}  \frac{1}{2}|\nabla U(t,X_t)|^2 -\frac{1}{2}\Delta U(t,X_t) - \partial_t U(t,X_t) dt \right) \\ 
 \exp\left(-U(1,X_1)+U(0,X_0) - \int_{-T}^{T}  \scrU(t,X_t) dt \right) \eeas
When passing to the density of the \textit{bridge} $P^{x,y}_{-T,T}$ against the Brownian bridge it is reasonable to expect that the term $-U(1,X_1)+U(0,X_0) $ cancels, since it depends only on the endpoints of the trajectory. Indeed it is known that (see \cite[Sec.5]{LeKre93}):
\[  \frac{dP^{x,y}_{-T,T}}{dW^{x,y}_{-T,T}} = Z^{-1}\exp\left(- \int_{-T}^{T}  \scrU(t,X_t) dt \right) \]
where $W^{x,y}_{-T,T}$ is the Brownian bridge and $Z$ a normalization constant.
Such consideration points at the fact that  $ \nabla \scrU$ is the right object to look for a quantitative study of bridges. In fact, we obtain the following result.
\begin{theorem}[Convexity of $\scrU$, couplings and gradient estimate]\label{thm:grest}
Let $U$ be time homogeneous. The following are equivalent for $\alpha>0$:
\begin{enumerate}[(i)]
\item $\scrU$ is $\frac{\alpha^2}{2}$ convex. That is,
\be\label{eq:convass} 
\inf_{z,v \in \RD: |v|=1}\nabla^2 \mathscr{U}(z)[v,v] \geq \alpha^2.
\ee
\item For any $T>0,t \in [-T,T]$ and any smooth function $f$ the following gradient estimates hold:
\be\label{eq:grest}
\forall x \in \RD, \quad | \nabla_x E_{P^{x,y}_{-T,T}} ( f( \omega_t)  ) | \leq \frac{\sinh(\alpha(T-t)) } {\sinh ( 2 \alpha T )} E_{P^{x,y}_{-T,T}} | \nabla f(\omega_t) |
\ee
\be\label{eq:gresttimerev}
\forall y \in \RD, \quad | \nabla_y E_{P^{x,y}_{-T,T}}( f( \omega_t)  ) | \leq \frac{\sinh(\alpha (T+t)) } {\sinh (2 \alpha T )} E_{P^{x,y}_{-T,T}} | \nabla f(\omega_t) |
 \ee
 \item For any  $ x_1,y_1,x_2,y_2 \in \RD$ there exists  a probability space $(\tilde{\Omega},\tilde{\cF},\tilde{P})$ and maps $X^{i}, i=1,2$ defined on it with the property that $X^{i} \# \tilde{P} = P^{x_i,y_i}_{-T,T}$ and 
\bes
\tilde{P}  \left( | X^{1}_t - X^{2}_t| \leq  \frac{\sinh(\alpha(T-t))}{\sinh(2\alpha T)} |x_2-x_1| + \frac{\sinh(\alpha (T+t) ) }{ \sinh(2\alpha T)} |y_2-y_1|  \quad \forall \, t \in [-T,T] \right) =1
\ees
 \end{enumerate}
\end{theorem}
In the statement of the Theorem and in the rest of the paper we use $\#$ for the push forward of probablity measures and we write $\nabla^2f(t,z)[v,v]$ for $\partial_v \partial_v f(t,z)$. \\

The main difficulty in dealing with bridges is that even though it is well known that the bridge is itself a Langevin dynamics, its time-dependent potential $U^{y}_T$ cannot be computed in closed form and one then has to choose a convenient representation to work with. One of the novelties of this article is a new representation of the bridge-potential $U^y_T$, where the role of $\scrU$  is most transparent. This is not the case for the well known $h$-transform decomposition (which goes back to Doob \cite{Doob1957}). The idea behind our representation is the following: instead of looking how to modify $-\nabla U$ to get the bridge drift $-\nabla U^y_T$, which is the $h$-transform approach, we look how to correct the drift of the \textit{Brownian bridge} to get $- \nabla U^y_T$.
The advantages of doing this is that it allows to bound first and second derivatives of $U^y_T$ using information on the corresponding derivatives of $\scrU$ by means of functional inequalities, namely Pr\'ekopa Leindler inequality. We shall use the bound on the first derivative to derive pathwise comparison principles for the bridges of different Langevin dynamics. 
Concerning the second derivative of $\scrU$,  we show at Theorem \ref{thm:grest} that convexity bounds on $\scrU$ are equivalent to gradient estimates along bridge semigroups with respect to both the initial and final point and to the existence of coupling between bridges with different endpoints in which the distance between trajectories decreases fastly. This result has to be compared with the well-known equivalence between the $\Gamma_2$ condition of Bakry and \'Emery, gradient estimates and couplings, see e.g. \cite[Thm 3.2.3]{bakry2013analysis} and \cite[Thm 1 and Cor 2]{von2005transport}.
In Theorem \ref{thm:logsob} we also show how the convexity of $\scrU$ implies a Logarithmic Sobolev inequality on path space. Such inequalities have been investigated on the setting of Lie groups in \cite{gross1991logarithmic}, and they often involve some extra "potential terms"( other authors have also investigated "heat kernel" logarithmic Sobolev inequalities on Loop groups, see \cite{driver1996logarithmic}). In our simpler setting, by choosing an appropriate norm for the Malliavin derivative, we can get rid of potential terms. In Corollary \ref{cor:conc} we deduce from this inequality some useful concentration of measure bounds for bridge measures, which were already partially obtained in \cite{conforti2016fluctuations}.
	We also investigate a natural notion of invariant measure for bridges, related to the concept of Gibbs measure on path space, in the spirit of \cite{priouret1975processus},\cite{doss1978processus} (see \cite{lorinczi2011feynman} for a recent account), which is to look for the limit as $T \uparrow +\infty$ of the marginal law at $t=0$ of $\LBR$. We provide a criterion for existence which, interestingly enough, shows that such an invariant measure may exist even in cases when the non-pinned process does not admit one. Contraction estimates in Wasserstein distance for the convergence to the invariant measure are also proved.
	The last question we address is that of quantifying how close are the bridge on an arbitrary Langevin dynamics and the Brownian bridge. We answer this question using the principles of Stein's method: in particular, to construct and solve Stein's equation we use the fact that bridges are the invariant measure of certain SPDEs (as pointed out in \cite{barbour1990stein},\cite{reznikoff2005invariant},\cite{hairer2007analysis}).
Let us give an overview of the main results beyond Theorem \ref{thm:grest}. We make their proof in Section 2.

\subsection{Main assumptions}  Let us first detail the main assumptions on $U$ and recall few basic notions. Our main object of study are bridges. The bridge of the Langevin dynamics between $x$ and $y$  over the time-interval $[-T,T]$ is the measure on the space $\Omega^{x,y}_{-T,T} := \{ \omega \in C([-T,T];\RD) :\omega_{-T} =x,\omega_{T} =y\}$  obtained by conditioning:
\bes
\forall A \, \text{measurable}\, A \subset \Omega^{x,y}_{-T,T}, \quad \LBR( \omega \in A) = P^x_{-T}(\omega \in A | \, \omega_{T} =y)
\ees
where $P_{-T}^x$ is the law of the Langevin dynamics whic starts in $x$ at $-T$.
We assume throughout the whole paper that  $U$ is four times countinuously differentiable in the space variable and continously differentiable in the time variable. Moreover it holds that
\be\label{eq:H}
\lim_{|z| \rightarrow +\infty} |U(-T,z)|=\lim_{|z| \rightarrow +\infty} |U(T,z)|=+\infty, \quad  \inf_{t \in [-T,T], z \in \RD} \scrU(t,z) > -\infty \tag{\textbf{H}}
\ee
A minor modification of the proof of Theorem 2.2.19 of \cite{royer2007initiation} shows  that under the current hypothesis, the solution of \eqref{eq:LD} exists and does not explode almost surely. Moreover, $P^x_{-T}$ admits a transition density which is everywhere positive, so that the bridges are well defined for any pair $x,y$ and not just in the almost-sure sense.

\subsection{A new representation of the drift}
It is well known that $P^{x,y}_{-T,T}$ is also a Langevin dynamics, i.e. there exist a time-dependent potential $U^{y}_T$  such that the law of:
\be\label{eq:BR}
dX_{t} = -\nabla U^{y}_{T}(t,X_t)dt + dB_t, \quad X_{-T} =x \tag{\textbf{BR}(U)}
\ee
is $\LBR$. According to the $h$-transform method, $U^{y}_{T}$ admits the following representation (see \cite[Th. 2]{Jam75}):
\bes
\begin{cases}
U^{y}_{T}(t,z) = U(t,z)- \log h(t,z) \\
\partial_t h(t,z) +\frac{1}{2} \Delta h(t,z) -\nabla U \cdot \nabla h(t,z) =0 \\ \lim_{t \uparrow T} h(t,z) = \delta_{y} 
\end{cases}
\ees
 This representation tells how one shall modify the drift of $P^x_{-T}$ to get the drift of $P^{x,y}_{-T,T}$. If we adopt a different viewpoint, i.e. we ask how one should correct the drift of the \textit{Brownian bridge} $W^{x,y}_{-T,T}$ to get the drift of $P^{x,y}_{-T,T}$, we arrive at a new representation formula, where the role of $\nabla \scrU$ is very clear. To state the Lemma we introduce  some notation, which we use throughout the article: $|\cdot|$ stands for the Euclidean norm on $\RD$, and $W^{z,y}_{t,T}$ is the Brownian bridge which starts at $t$
in $z$ and ends at $T$ in $y$.
\begin{lemma}\label{genhtrans}
$P^{x,y}_{-T,T}$ is a Langevin dynamics, whose drift $- \nabla U^{y}_{T}$ is given by:
\be\label{eq:driftreprs}
- \nabla U^{y}_{T}(t,z)  = -\nabla H^y_T(t,z)+ \nabla \log \psi(t,z) 
\ee
with
\be \label{def:h}
H^y_T(t,z) = \frac{|y-z|^2}{2(T-t)}
\ee 
and
\be\label{def:psi} 
 \psi(t,z)= E_{W^{z,y}_{t,T}} 
\left( \exp \Big(-\int_{t}^{T}  \mathscr{U}(s,\omega_s) ds \Big) \right)  
\ee
\end{lemma} 

\begin{remark}
\begin{itemize}
\item  Even though $\psi$ depends on $\scrU$ itself, the drift of $\LBR$, $-\nabla U^y_T$, depends only on $\nabla \scrU$. One can see this by plugging \eqref{eq:driftreprs} $\scrU+c$ into  and seeing that the \eqref{eq:driftreprs} does not depend on $c$.
\item $-\nabla H^y_T$ is indeed the drift of the Brownian bridge
\end{itemize}
\end{remark}
\subsection{Convexity of $\scrU$, couplings and gradient estimates}
 The main result is Theorem \ref{thm:grest} above, where we show the equivalence between convexity bounds for $\scrU$, a gradient estimate along the time-inhomogenous bridge semigroup, and the existence of certain couplings between bridges with different endpoints. In the gradient estimate, we consider both perturbations of the initial and final state of the bridge. Some remarks are in order.
\begin{remark}
\begin{itemize}
\item
Under the hypothesis of Theorem \ref{thm:grest}, we prove in Theorem \ref{theorem:invmeas} that there exists a potential $V$ satisfying $\nabla^2V(z)[v,v] \geq \alpha $ and such that the associated Langevin dynamics, which we denote $Q^x_{-T}$, share the bridges with $P^x_{-T}$, i.e. \[ \forall T>0, \, x,y \in \RD \quad Q^{x,y}_{-T,T} = P^{x,y}_{-T,T}.\] The classical Bakry \'Emery gradient estimate for $Q^x_0$ reads as 
\be\label{eq:grestbak} \forall x \in \RD, \quad | \nabla_x E_{Q^{x}_{0}} ( f( \omega_t)  ) | \leq \exp(-\alpha t) E_{Q^{x}_{0}} | \nabla f(\omega_t) | \ee
It is worth noticing that we can obtain the same result by using Theorem \ref{thm:grest}. Fix an arbitrary point $y$ and $t>0$. By riparametrizing time (i.e. by considering bridges over $[0,2T]$ instead $[-T,T]$) we can rewrite the estimate \eqref{eq:grest} as:
\[\forall x \in \RD, \quad | \nabla_x E_{Q^{x,y}_{0,2T}} ( f( \omega_t)  ) | \leq \frac{\sinh(\alpha( 2T - t)) } {\sinh ( 2 \alpha T )} E_{Q^{x,y}_{0,2T}} | \nabla f(\omega_t) |
  \]
It is easy to see that, as $T \rightarrow +\infty$ we have that $\frac{\sinh(\alpha( 2T - t)) } {\sinh ( 2 \alpha T )}$ converges to $\exp(-\alpha t)$.  Moreover, a slight modification of the proof of \eqref{eq:invconvest} in Theorem \ref{theorem:invmeas} shows that
\be\label{eq:smgrpconv} \forall g \ \text{bounded and measurable}, \quad  \lim_{T \rightarrow + \infty}  E_{Q^{x,y}_{0,2T}}(g (\omega_t))= E_{Q^{x}_{0}}( g(\omega_t) ) \ee
Combining these two observations, we get back the usual gradient estimate \eqref{eq:grestbak}. Therefore, our result is consistent with the Bakry \'Emery theory, when the two sets of hypothesis intersect. 

\item It is rather tempting to say that \eqref{eq:smgrpconv} is valid for any potential $V$, and not just for convex ones.
 The intuition behind this should be that when $T$ is very large the conditioning $\{X_{2T}=y\}$ has almost no effect at time $t$. However this is false in general. A careful look at the proof of  Theorem \ref{theorem:invmeas} shows that it is true 
 only if the non-pinned Langevin dynamics associated with $V$ admits an invariant measure.
 
 \item The equivalence between point $(i)$ and $(ii)$ can be easily be extended to the time dependent case.
 \end{itemize}
\end{remark}
\subsection{Convexity of $\scrU$ and Logarithmic Sobolev inequality}
Here we present a Logarithmic Sobolev inequality on path space (LSI for short). For simplicity, we set $x=y=0$.
Let us consider the Hilbert space $H = L^2([-T,T],\RD) $, and denote by $\langle.,.\rangle_{H}$ the usual scalar product on it.
For $\alpha>0$ we introduce the scalar product $\langle.,.\rangle_{\alpha}$ on $H$ defined by
$\langle h,g \rangle_{\alpha}:= \langle \varphi,g \rangle_H$, where $\varphi$ solves
\be\label{eq:ODEforh}
\varphi_t -\alpha^2 \int_{-T}^t\int_{-T}^s \varphi_r dr ds +  \frac{\alpha^2}{2T} \int_{-T}^{T}\int_{-T}^v\int_{-T}^u \varphi_r dr du dv  = h_t -\frac{1}{2T}\int_{-T}^{T} h_u du , \quad t \in [-T,T]
\ee
Note that since $h$ and $\varphi$ are $\RD$ valued, the equation above has to be understood componentwise.
When $h$ is smooth, $\varphi$ is the solution of the ODE
\bes
\begin{cases} \ddot{\varphi}_t - \alpha^2 \varphi_t = \ddot{h}_t \\
\varphi_{-T} =h_{-T} \\
\int_{-T}^{T}\varphi_t dt =0 
\end{cases}
\ees

Next we say, as usual, that $F : \Omega^{0,0}_{-T,T} \rightarrow \RD$ is a simple functional if
\bes
F(\omega) = f\Big( \int_{-T}^{T} h^1_t \cdot d\omega_t,..,\int_{-T}^{T} h^n_t \cdot d \omega_t \Big)
\ees
for a smooth function $f$ with bounded derivatives and finitely many $h^1,..,h^n\in H$.
Its (Malliavin) derivative  is the $H$-valued random variable:
\be\label{def:Mallder}
DF(\omega) = \sum_{i=1}^n \partial_{i}f\Big(\int_{-T}^{T} h^1_t \cdot d\omega_t,..,\int_{-T}^{T} h^n_t d \cdot \omega_t \Big) h^i
\ee
where $\partial_if$ stands for the derivative of $f$ w.r.t. the i-th coordinate .

We obtain the following result.
\begin{theorem}\label{thm:logsob}
 Let $\scrU$ satisfy the convexity assumption \eqref{eq:convass}. Then $P^{0,0}_{-T,T}$ satisfies the following inequality:
\be\label{logsob}
E_{P^{0,0}_{-T,T}}(F \log F) - E_{P^{0,0}_{-T,T}}(F)\log E_{P^{0,0}_{-T,T}}(F) \leq 2\, E_{P^{0,0}_{-T,T}}\Big(\frac{1}{F}\langle D F , DF \rangle_{\alpha}\Big)
\ee
for any positive simple functional $F$. 
\end{theorem}
\begin{remark}
It is interesting to compare this inequality with the LSI on path for Brownian motion, see e.g. \cite{capitaine1997martingale} in the case $\alpha=0,d=1$.
For the non pinned Brownian motion, the Mallivain derivative operator is the one that we would get by using the scalar product
$$\langle h,g\rangle_H  = \int_{-T}^{T}  h_t  g_t dt$$ 
and then repeating the construction above.
For the Brownian bridge it follows from a simple calculation that the scalar product $\langle \cdot,\cdot \rangle_0$ i
$$ \langle h,g \rangle_0  = \int_{-T}^{T} h_t  g_t dt - \int_{-T}^{T} h_t dt \int_{-T}^{T} g_t dt $$
\end{remark}
A consequence of the inequality is the following concentration of measure estimate which has already been obtained in \cite[Thm 2.1]{conforti2016fluctuations} with different methods.
\begin{cor}\label{cor:conc}
In the hypothesis of Theorem \ref{thm:logsob} the following estimate holds for any 1-Lispchitz function, $t \in [-T,T]$ and $R>0$:
\bes
P^{0,0}_{-T,T}\Big(f(\omega_t) - E_{P^{0,0}_{-T,T}}(f(\omega_t)) \geq R \Big) \leq \exp \Big(-\xi_{\alpha}(t) R^2 \Big)
\ees
with
\bes
\xi_{\alpha}(t)=\frac{  \alpha \sinh(2\alpha T)}{2 \sinh(\alpha(T+t)) \sinh(\alpha(T-t))}
\ees
\end{cor}
The time evolution of the concentration parameter $\xi_{\alpha}(t)$ reflects the fact that bridges fluctuate the most around $t=0$, as the influence of pinning is at its minimum whereas close to the endpoints it becomes so strong that fluctuations die out.

\subsection{Reciprocal class}
the natural collocation of some of our results is within the framework of reciprocal processes. The reciprocal class associated with $U$, which we call $\mathfrak{R}_T(U)$ is the set of all bridge mixtures.
For a given probability measure $\mu$ on $\RD \times \RD$ we  define $P^{\mu}_{-T,T}$ via
\be\label{eq:PMU}
P^{\mu}_{-T,T}(A) := \int_{\RD \times \RD} P^{x,y}_{-T,T}(A \cap \{X_{-T}=x,X_{T}=y\}) \mu(dx dy)
\ee
Note that when $\mu = \delta_x \otimes \delta_y$ we obtain the $xy$ bridge. 
The reciprocal class is then:
\bes
\mathfrak{R}_T (U):= \left\{Q \in \cP(\Omega_{-T,T}): Q= \PMU \ \text{for some} \ \mu \in \cP(\RD \times \RD) \right\}.
\ees
where $\cP(\cdot)$ stands for the space of probability measures and $\Omega_{-T,T} = \bigcup_{x,y \in \RD} \Omega^{x,y}_{-T,T}$. Even though $\mathfrak{R}_T (U)$ is mostly made of non-Markov processes, it still contains infinitely many Markovian elements. For example, it contains all Langevin dynamics relative to $U$, and all of their bridges. But in general it contains more: precisely another Langevin dynamics relative to a different potential $\tilde{U} $ belongs to $\frak{R}_T(U)$ if and only if $\nabla \scrU = \nabla \tilde{\scrU}$. This is due to Krener \cite{Kre88}. All elements of a reciprocal process enjoy a weaker property than the Markov property: they are time Markov fields, see \cite{LRZ}.
\begin{myxmpl}
Let $U(z)= \frac{\alpha}{2}|z|^2$ for $\alpha>0$. Then $\mathfrak{R}_T(U)$ contains the Ornstein Uhlenbeck process with mean reversion, i.e. with drift $- \alpha z$ but also the Orntein Uhlenbeck process with mean repulsion, i.e. with drift $\alpha z$. Furthermore, it contains interesting strictly non Markovian processes. It is showed in \cite{RT02} that the periodic Ornstein Uhlenbeck process, i.e. the solution of
\bes
dX_{t} = -\alpha X_t + dB_t, \quad X_{-T} = X_{T}
\ees
belongs to $\mathfrak{R}_T(U)$.
\end{myxmpl}

\subsection{Comparison principle}
We fix two potentials $U^1$ and $U^2$ and denote the elements of the corresponding reciprocal classes $P^{i,\mu}_{-T,T}, i=1,2$. In the same way we define $P^{i,x,y}_{-T,T}, i=1,2$ . $\scrU^i$ is defined as in \eqref{def:reciprocalpot}. The following comparison principle is obtained by combining the representation formula \eqref{eq:driftreprs} with the standard comparison principles for SDEs.
\begin{theorem}\label{thm:comparisonprinciple}
Let $d=1$ and $x,y$ be fixed. If $U^i, i=1,2$ are such that 
\bes
\forall t \in [-T,T], \, z \in \R, \quad (\scrU^1 - \scrU^2)'(t,z) \geq 0 .
\ees 	
Then for any  $ \mu \in \cP(\R \times \R)$ there exists  a probability space $(\tilde{\Omega},\tilde{\cF},\tilde{P})$ and maps $X^{i}, i=1,2$ defined on it with the property that $X^{i} \# \tilde{P}= P^{i,\mu}_{-T,T}$ and
\bes
P\left(X^{2}_t \geq X^{1}_t \ \forall t \in [-T,T]\right)=1
\ees
\end{theorem}
In particular we have the following corollary concerning bridges.
\begin{cor}\label{cor:comparison}
For any  $x,y \in \R$ there exists  a probability space $(\tilde{\Omega},\tilde{\cF},\tilde{P})$ and maps $X^{i}, i=1,2$ defined on it with the property that $X^{i} \# \tilde{P}= P^{i,x,y}_{-T,T}$ and
\bes
P\left(X^{2}_t \geq X^{1}_t \ \forall t \in [-T,T]\right)=1
\ees
\end{cor}
\subsection{Bounding the distance between bridges and reciprocal processes}
Here, taking advantage of Stein's method techniques, we tackle the question of estimating the  $W_1$ distance between  of $\LBR$ and the Brownian bridge $\LBBR$ and, more generally, of bounding the distance between $\PMU$ and $W^{\mu}_{-T,T}$ for an arbitrary $\mu$.  We equip $\OMB$ with the sup-norm $\|\cdot \|_{\infty}$ and say that a functional $f: \Omega \rightarrow \R$ is $1$-Lipschitz if and only if
$ |f(\omega')-f(\omega) | \leq \| \omega - \omega' \|_{\infty} $. The $W_1$ distance is:
\bes
W_1(P,Q) = \sup_{f \, \text{1-Lipschitz} } |E_P(f)-E_Q(f)|
\ees
The fact that bridges may be viewed as invariant measures of SPDEs (see \cite{barbour1990stein}, \cite{reznikoff2005invariant},\cite{hairer2005analysis} and \cite{hairer2007analysis}), offers a natural candidate for the Stein equation, which we use to obtain the following result.
\begin{theorem}\label{thm:Stein}
Let $U$  be such that 
\bes
\sup_{t \in [-T,T],z \in \RD}|\nabla \scrU(s,z)|:=\|\nabla\scrU \|_{\infty} < + \infty 
 \ees
Then we have:
\be\label{eq:steinreciprocalbound}
W_1(P^{\mu}_{-T,T},W^{\mu}_{-T,T})\leq C T^2 \|\nabla \scrU \|_{\infty} 
\ee
where
  $$ C= E_{W^{0,0}_{-1,1}} \Big(\| \omega_s \|_{\infty} \int_{-1}^{1} |\omega_s| ds \Big)<+\infty  .$$
In particular:
\be\label{eq:steinbridgebound}
W_1(P^{x,y}_{-T,T},W^{x,y}_{-T,T}) \leq C T^2\|\nabla \scrU \|_{\infty}  
\ee
\end{theorem}
Let us note that similar bounds can be obtained by working with different norms instead of $\|\cdot\|_{\infty}$, with minor changes from the proof of Theorem \ref{thm:Stein}. This result can be viewed from a broader perspective as one of the instances of a more general method to compare conditional probabilities, which is developed in the forthcoming article \cite{chiarini2016approximating}. 
\subsection{Invariant measure and contraction estimates.}
Here we propose a natural notion for invariant measure associated with the family $\{\mathfrak{R}_T(U)\}_{T >0}$.
\begin{mydefn}\label{def:invmeas}
 $m$ is an invariant measure if for any $\mu \in \cP(\RD \times \RD)$, $ \omega_0 \#\PMU$ converges weakly to $ m$ as $T \rightarrow + \infty$.
\end{mydefn}
Using a kind of Girsanov Theorem for bridges proved in \cite[Sec. V]{LeKre93}, it is possible to view $\LBR$ as a penalization (in the sense of \cite[Sec 1.4]{roynette2009penalising}) by an additive functional of the Brownian bridge. From this angle, the notion of invariant measure we have just introduced is included in that of Gibbs measure on path space, see \cite[Chap. 4]{lorinczi2011feynman} and references therein. 
In the next theorem we establish existence when $\scrU$ is convex and we address the problem of studying how fast $\omega_0 \# P^{\mu}_{-T,T}$ approaches $m$. We denote by $W_p(\cdot,\cdot)$ the p-Wasserstein distance on $\cP(\RD)$ and by $\tilde{W}_p(\cdot,\cdot)$ the p-Wasserstein distance on $\cP(\RD \times \RD)$ (see \eqref{def:pwasser} for precise definitions).
\begin{theorem}\label{theorem:invmeas}
 Let $U$ be time homogeneous and $\mathscr{U}$ satisfy the convexity assumption \eqref{eq:convass}. Then:
 \begin{enumerate}[(i)]
 \item There exist an invariant measure $m$ such that:
 \be\label{eq:invconvest}
 \frac{d m}{d \lambda }(z) = \exp(-2V(z)), \quad \inf_{v,z \in \RD : |v|=1}\nabla^2 V(z)[v,v] \geq \alpha .
 \ee
 where $\lambda$ is the Lebesgue measure on $\RD$.
 \item
For any $T>0$, and any $\mu,\nu \in \cP(\RD \times \RD)$:
\be\label{eq:contrest2}
W_p(\omega_0 \# \PMU,\omega_0 \# P^{\nu}_{-T,T}) \leq \frac{1}{\sqrt{2}\cosh(\alpha T)} \tilde{W}_p(\mu,\nu)  
\ee	
In particular:
\be\label{eq:contrest}
W_p(\omega_0 \# \PMU,m) \leq \frac{1}{\sqrt{2} \cosh(\alpha T)} \tilde{W}_p(\mu,(\omega_{-T},\omega_{T})\# \hat{P}^m)  
\ee
where $\hat{P}^m$ is the stationary Markov process on $C(\R,\RD)$ whose generator is $\scrL$ (see \eqref{eq:LDgen}).
\end{enumerate}
 \end{theorem}
As $T \uparrow + \infty$, $\tilde{W}_p(\mu,(\omega_{-T},\omega_{T})\#\hat{P}^m)$  converges to $\tilde{W}_p(\mu,m \otimes m)$, so the bound \eqref{eq:contrest} really implies an exponential decay  of $W_p(\omega_0 \# \PMU,m) $. Moreover,  $\tilde{W}_p((\omega_{-T},\omega_{T})\# \hat{P}^m, m \otimes m)$ can also be quantitatively estimated using the log-concavity of $m$, which is granted by \eqref{eq:invconvest}.
It is possible to derive the following
\begin{cor}
In the hypothesis of the Theorem, we have:
\[\tilde{W^p}\Big( (\omega_{-T},\omega_{T})\#\hat{P}^m,  m \otimes m \Big) \leq \exp(-2 \alpha T) \int_{\R^d}\Big(\int_{\R^d} (x-y)^p m(dy)\Big)^{1/p} m(dx) \]   
\end{cor}

 \begin{remark}
  \begin{itemize}
 \item
The existence part of Theorem \ref{theorem:invmeas} can be proven under the weaker assumption that $\mathscr{U}$ is Kato decomposable using the results of \cite[Sec. 4.1,4.2]{lorinczi2011feynman}. To be self contained we prefer not to rely on them and give a direct proof under the convexity assumption.
\item In the language of \cite[Sec. 3.10.2]{lorinczi2011feynman} the process $\hat{P}^m$ is the $P(\phi)_1$ process for the potential $\scrU$.

\item As it is clear from the proof of the Theorem, if $U$ is itself convex, the invariant measure is the same as for the non-pinned dynamics.
\end{itemize}
\end{remark}
\begin{myxmpl}
Let $\alpha>0$. The above Theorem tells, unsurprisingly, that if $U=  \frac{\alpha}{2} |z|^2$, then $ \mathfrak{R}_{T}(U)$ admits an invariant measure, since $\scrU = \frac{\alpha}{2} (\alpha |z|^2-1)$. However, as we have seen in the previous example, the process with mean repulsion, whose associated potential is $-\frac{\alpha}{2} |z|^2$ lies also in $\mathfrak{R}_{T}(U)$. Therefore the bridges of the Ornstein Uhlenbeck process with mean repulsion admit an invariant measure, whereas the non conditioned dynamics clearly doesn't.
\end{myxmpl}
\subsection*{Some heuristics}
The reciprocal characteristic has been interpreted as a ``mean acceleration" term in the context of Stochastic mechanics , see e.g. \cite{Nel67}. Although this article does not aim at establishing connections with physics, such interpretation is useful to better understand our probabilistic results: if one considers the one dimensional ODE
\bes
\begin{cases} \dot{x}_t = -  U'(t,x_t) \\ x_{-T} = x \end{cases}
\ees
then the acceleration of $x_t$ is $\ddot{x}_t = U''(x_t)U'(x_t)-\partial_tU'(t,x_t)  $. When adding to the ODE a ``$dB_t$" term to get the Langevin dynamics, one should think that the Brownian motion produces a change in the \emph{mean acceleration} by $-\frac{1}{2} U'''$, so  that the mean acceleration of the Langevin dynamics is $\scrU'$. 
Such intuition has been developed by several authors, see for example \cite{Th93},\cite{Kre97}, and the forthcoming work \cite{conforti2017coupling} (where bridges on Riemannian manifolds are also considered).  
Adopting a statistical viewpoint yields another interesting interpretation. Indeed, Krener's observation can be rephrased by saying that there is no statistical algorithm that allows to infer the potential $U$  of a Langevin dynamics from the bridge sample paths. More precisely, one can extract from bridge-samples the reciprocal characteristic $\nabla \scrU$, and no more. This is just one manifestation of the fact that  conditional distributions carry less information than the full measure.\\
Theorem \ref{thm:comparisonprinciple} shows that if $U^1$ and $U^2$ are such that ${\scrU^1}' \geq {\scrU^2}'$, one can couple the $xy$ bridges of the corresponding Langevin dynamics in such a way that the bridge associated with $U^2$ lies always above the bridge associated with $U^1$ . This results establishes a nice parallelism with what is known for ODEs. Indeed, consider two deterministic particles $x^1_t$ and $x^2_t$ with constant acceleration $u_1$, $u_2$ with $ u_1 \geq u_2$and identical endpoints:
\bes
\begin{cases}
\ddot{x}^1_t=u_1 \\
x^1_{-T}=x , \quad x^{1}_T=y 
\end{cases}\quad
\begin{cases}
\ddot{x}^2_t=u_2 \\
x^2_{-T}=x , \quad x^{2}_T=y 
\end{cases}
\ees
Then it is easy to see, by solving the ODE explicitly that $x^1_t\leq x^2_t$ forall $-T \leq t \leq T$.  Recalling the interpretation of $\scrU'$ as a stochastic mean acceleration, our Theorem \ref{thm:comparisonprinciple} can be viewed as a stochastic version of the particular type of  comparison principle for ODEs we have just outlined in the lines above.\\
 On the same spirit, one obtains an interpretation of Theorem  \ref{thm:grest}. Consider two deterministic one dimensional particles $x^1_t,x^2_t$ subject to the same positional acceleration field of the form $ \scrU' (\cdot)$, for some function $\scrU$. The particles have the same final position $y$ but different starting points $x^1,x^2$ with $x^2 > x^1$.
That is,
\bes
\begin{cases}\ddot{x}^1_t =  \scrU' (x^1_t), \\ x^{1}_{-T}=x_1, \, x^{1}_T =y \end{cases}\quad \begin{cases}\ddot{x}^2_t =  \scrU' (x^2_t),\\  x^{2}_{-T}=x_2, \, x^{2}_T =y \end{cases}
\ees
Then we have $x^2_t \geq x^1_t$ at any time. Let us now study $|x^2_t - x^1_t| = x^2_t - x^1_t$. It is easy to see that if $\scrU$ is $\frac{\alpha^2}{2}$ convex, i.e. $\scrU'' \geq \alpha^2$ we have
 \[ \frac{d^2}{dt^2}(x^2_{t} - x^1_t) = \scrU'(x^2_t)-\scrU'(x^1_t) \geq  \alpha^2 (x^2_t-x^1_t). \] 
 Integrating this differential inequality for $x^2_t-x^1_t$ one obtains:
\bes
\forall -T \leq t \leq T \quad  |x^2_t- x^1_t| \leq |x^2 - x^1| \frac{\sinh(\alpha(T-t)) } {\sinh ( 2 \alpha T )}
\ees
 It is natural to wonder whether the same reasoning applied to the ``mean stochastic acceleration" produces the same results. Quite surprisingly it does, as point (iii) of Theorem \ref{thm:grest} shows. (Note that there also different final conditions  are taken into account). Again, we stress that the analogy with mechanics has the sole purpose of helping in understanding the  mathematical results, and we make no physical claim.\\

\subsection{Some remarks about reciprocal processes}As we have seen, some of the results take their most complete form when applied to the family of \textit{reciprocal processes} associated with a Langevin dynamics. They are constructed as arbitrary bridge-mixtures: therefore they are not Markov in general but rather time Markov-fields (see \cite[Sec. 2]{LRZ}). These processes were introduced by Bernstein \cite{Bern32}, who took the inspiration from Schr\"odinger's works \cite{Schr},\cite{Schr32} on the analogies between diffusion processes and quantum dynamics. The series of paper by Jamison \cite{Jam70},\cite{Jam74},\cite{Jam75} initiated a systematic investigation of their structural properties, and laid the foundations for their application in the context of Stochastic Mechanics, see for example \cite{LeKre93},\cite{Zam86},\cite{CruZa91},\cite{TZ97} and references therein. In this paper we add a new motivation for their study: we construct a natural class of minimization problems which generalizes the well known Schr\"odinger problem (see the recent survey \cite{LeoSch}) and whose solution is a reciprocal process.\\
This result generalizes the well known result due to F\"ollmer \cite{FOLL88} that the solution to the Schr\"odinger problem is one of the \textit{Markovian} elements of the reciprocal class.

\begin{prop}\label{label:normapprox}
Let $c:\R \rightarrow \R_{\geq 0}$ be strictly convex. We define the cost function $C(.,.)$ on $\cP(\Omega_{-T,T})$:
\bes
C(Q,P) = \begin{cases} E_{P}\Big( c\Big(\frac{dQ}{dP} \Big) \Big), \quad & \mbox{if $Q\ll P$} \\ 
+\infty, \quad & \mbox{otherwise}
\end{cases}
\ees
Consider the problem
 \be\label{approxprob}
 \min C(Q,P^{\mu}_{-T}), \quad Q \in \cP(\Omega_{-T,T}), \, {\omega_{-T}\#}Q = \mu, \, {\omega_T \#}Q = \nu
 \ee
 where $P^{\mu}_{-T}$ is the Langevin dyanmics for $U$ over $[-T,T]$ with initial distribution $\mu$. Then, if a solution $\hat{Q}$ exists, it is in $\mathfrak{R}_{T}(U)$. 
\end{prop}
The Schr\"odinger problem is obtained with $c(z)=z \log z$. But other choices are interesting: $c(z)=z^2$ yields the $L^2$-distance when $Q \ll P$, and $c(z)= z^{1/2}$
the Hellinger distance when $Q \ll P$.

	\vspace{1cm}
\section{Proofs of the results}

\subsection{Proof of Lemma \ref{genhtrans}}
In fact, we prove a stronger form of it, which we are using to prove the next results. We fix two potentials $U^1$ and $U^2$, call $P^{1,x,y}_{-T,T}$,$P^{2,x,y}_{-T,T}$ the corresponding bridges and $U^{1,y}_T$, $U^{2,y}_T$ the bridge-potentials.
\begin{lemma}\label{genhtrans2} The following decomposition holds:
\be\label{eq:driftreprs2}
- \nabla U^{2,y}_{T}(t,z)  = -\nabla U^{1,y}_T(t,z)+ \nabla \log \psi(t,z) 
\ee

with 

\be\label{def:psi2} 
 \psi(t,z):= E_{P^{1,z,y}_{t,T}} 
\left( \exp \Big(-\int_{t}^{T}  \scrU^2-\scrU^1 (s,\omega_s) ds \Big) \right)  
\ee
\end{lemma} 
\begin{proof}
  In \cite[Sec V]{LeKre93} and \cite[Sec. 3]{conforti2016fluctuations} a Girsanov Theorem for reciprocal processes, and hence for bridges, is given. Adapted to our notation, it gives:  
\bes
  \frac{dP^{2,x,y}_{-T,T}}{ d P^{1,x,y}_{-T,T} } = \frac{1}{ \psi(-T,x)} \exp \left( -\int_{-T}^{T} \bar{\scrU}(s,\omega_s) ds\right): = M
  \ees
  where $\psi$ is given at \eqref{def:psi2} and $\bar{\scrU} = \scrU^2-\scrU^1$. Therefore:

\be\label{eq:Mt}
M_t := E_{P^{1,x,y}_{-T,T}} [M \vert \cF_t] = \exp \Big( - \int_{-T}^{t} \bar{\scrU}(s,\omega_s)ds  \Big) \frac{\psi(t,\omega_t)}{\psi(-T,x)}
\ee
where $\{\cF_T\}_{-T \leq t \leq T}$ is the canonical filtration on $\OMB$.
An application of the Feynman-Kac formula tells that $\psi$ solves:
\be\label{FKpde}
\partial_{t} \psi(t,z) + \frac{1}{2} \Delta \psi(t,z) -\nabla U^{1,y}_{T}\cdot \nabla \psi (t,z) - \bar{\scrU}(t,z)\psi(t,z) =0.
\ee
The desired boundary condition is easily derived from from \eqref{def:psi2}.
 Now we rewrite $M_t$ (defined in \eqref{eq:Mt}) as an exponential martingale. Using integration by parts, the PDE \eqref{FKpde}
 and recalling that under $P^{1,x,y}_{-T,T}$ the canonical process can be written as
  \bes
  \omega_t = - \int_{-T}^t \nabla U^{1,y}_T(s,\omega_s)ds + B_t
  \ees
 for some Brownian motion $(B_t)_{-T \leq t \leq T}$ , we obtain :
 \beas
d M_t &=&  \psi(t,\omega_t) d  \exp \Big(- \int_{-T}^{t} \bar{\scrU}(s,\omega_s) ds \Big) + \exp\Big(- \int_{-T}^{t} \bar{\scrU}(s,\omega_s) ds \Big)  d \psi(t,\omega_t)\\
&\stackrel{ \text{It\^o}}{=}& -\psi(t,\omega_t)  \exp\Big(- \int_{-T}^{t} \bar{\scrU}(s,\omega_s) ds \Big) \bar{\scrU}(t,\omega_t) dt \\
&+& \exp\Big( -\int_{-T}^{t} \bar{\scrU}(s,\omega_s) ds \Big)\{ \nabla \psi(t,\omega_t) \cdot d\omega_t + [\partial_t  \psi(t,\omega_t) + \frac{1}{2} \Delta   \psi(t,\omega_t) ] dt\} \\
&=&\exp\Big( -\int_{-T}^{t} \bar{\scrU}(s,\omega_s) ds \Big)\psi(t,\omega_t) \nabla \log \psi(t,\omega_t) \cdot dB_t  \\
&+& \exp\Big( -\int_{-T}^{t} \bar{\scrU}(s,\omega_s) ds \Big)\{   \partial_t \psi(t,\omega_t)+ \frac{1}{2} \Delta   \psi(t,\omega_t) -  \nabla U^{1,y} \cdot \nabla \psi(t,\omega_t) - \bar{\scrU} \psi (t,\omega_t)  \}dt \\
&\stackrel{\eqref{FKpde}}{=}& M_t \nabla \log \psi(t,\omega_t) \cdot dB_t 
 \eeas
Therefore $M_t$ is the exponential martingale associated with the martingale $ A_t = \int_{-T}^{t} \nabla \log \psi(t,\omega_t) \cdot dB_t$. Consider now a smooth function $f$ with bounded derivatives. We know that under $P^{1,x,y}_{-T,T}$ the process
\[
M^f_t:=f(\omega_t) - \int_{-T}^{t} \tilde{\scrL}_s(f)(\omega_s)ds
\]
is a local martingale, where 
\[ \tilde{\scrL}_s(f)(z) = \frac{1}{2} \Delta f(z) - \nabla U^{1,y}(s,z)\cdot \nabla f(z) \]
Girsanov's Theorem tells that under $P^{2,x,y}_{-T,T}$
\[
f(\omega_t) - \int_{-T}^{t} \tilde{\scrL}_s(f)(\omega_s)ds - \langle M^f,A\rangle_t
\]
is a local martingale. A standard computation using It\^o isometry yields
\[
\langle M^f,A\rangle_t = \int_{-T}^t \nabla \log \psi(s,\omega_s) \cdot \nabla f(\omega_s)ds
\]
from which the conclusion follows, since $P^{2,x,y}_{-T,T}$ solves the martingale problem associated with $ \tilde{\scrL}_s + \nabla \log \psi(s,\cdot) \cdot \nabla f$.
\end{proof}
\subsection{Proof of Theorem \ref{thm:grest}}%
We prove three preparatory Lemmas: the first one contains some explicit computations for Ornstein Uhlenbeck bridges.
\begin{lemma}\label{lemma:OUbridge}
Let $U$ be such that $\nabla \scrU = \alpha^2 z$. The following hold:
\begin{enumerate}[(i)]
\item
\[
\forall t \in [-T,T], z,v \in \RD:|v|=1 \quad \nabla^2 U^{y}_{T}(t,z)[v,v] = \alpha \coth(\alpha(T-t))
\]

\item The following identity in distribution holds for any $x,x',y \in \RD$:
\be\label{eq:OUbridgetranslation}
\forall F \, \text{ measurable} , \quad E_{P^{x',y}_{-T,T}}(F(\omega)) = E_{\LBR}(F( \omega + (x'-x)\xi_T ))
\ee
where
\bes
\xi_T(t):= \frac{\sinh(\alpha(T-t) )}{\sinh(2 \alpha T)} 
\ees
\item If we define $S_T: \Omega^{x,y}_{-T,T} \rightarrow \Omega^{x,y}_{-1,1}$ through $(S_T\omega)_t = \omega_{Tt}$ it holds that
\be\label{eq:timechange}
\lim_{T \downarrow 0} S_T \# \LBR =  \delta_{\varphi}
\ee
where $\varphi_t = \frac{(1-t)}{2}x+\frac{(1+t)}{2}y$ for $t \in [-1,1]$.
\end{enumerate}
\end{lemma}
\begin{proof}
\begin{enumerate}
\item[(i)] First, let us observe that if $ U = \frac{\alpha}{2} |z|^2$, then $\nabla \scrU = \alpha^2 z $
, so w.l.o.g $\LBR$ is the bridge of the OU process with mean reversion $\alpha$. The $h$-transform formula tells that
\bes
-\nabla U^y_T(t,z) = - \alpha z + \nabla_z \log p(T-t,z,y)
\ees
where $p(T-t,z,y)$ of the Ornstein-Uhlenbeck process. The Ornstein Uhlenbeck transition density is known explicitly. We have:
\bes
\log(p_{T-t}(z,y)) =  -\alpha\sum_{j=1}^d\frac{(y_j - \exp(-\alpha(T-t)) z_j )^2} {(1- \exp(-2 \alpha(T-t) ) ) } + c
\ees
where $c$ is some constant which does not depend on $z$.
With a direct computation one gets that:
\be\label{eq:OUdrift}
-\nabla U^y_T(t,z) = - \alpha [  \coth(\alpha(T-t))z - \frac{1}{\sinh(\alpha(T-t))} y ]
\ee
The conclusion then follows by taking another derivative.
\item[(ii)] 
Let $X_t$ be a solution to 
\bes
dX_t = - \nabla U^y_T(t,X_t)dt + dB_t, \quad X_{-T}=x
\ees
Then $Y := X + (x'-x) \xi_T$ satisfies
\bes
d Y_t =  [- \nabla U^y_{T}(t,Y_t-(x'-x)\xi_T(t))+(x'-x)\dot{\xi}_T(t) ] dt + dB_t, \quad Y_{-T}=x +(x'-x)\xi_{T}(-T)
\ees
Since $\xi_T(-T)=1$, the conclusion follows if we show that $- \nabla U^y_{T}(t,z-(x'-x)\xi_T(t))+(x'-x)\dot{\xi}_T(t) = \nabla U^y_T(t,z)$ everywhere. This fact can be checked with a direct computation using \eqref{eq:OUdrift}.
\item[(iii)] Using Brownian scaling and change of variables, an elementary computation shows that $S_T \omega \# \LBR$
is distributed as the law of
\bes
dX_t = T \nabla U^y_T(T s,X_s)ds + \sqrt{T} B_{t}, \quad X_{-1}=x,t \in [-1,1]
\ees
Observing that as $ \lim_{T \downarrow 0} T \nabla U^y_T(T t,z)=(y-z)/(1-t) $ and that the solution of the limiting ODE
\bes
\dot{\psi}_t = \frac{y-\psi_t}{1-t}, \quad \psi_{-1}=x
\ees
is precisely $\varphi$, the conclusion follows after some standard computations.
\end{enumerate}
\end{proof}
The next Lemma is a robust version of the convexity estimate of the former Lemma.
 \begin{lemma}\label{lem:conc}
Let \eqref{eq:convass} hold. Then 
\bes
\forall -T \leq t \leq T,  z \in \RD, \quad \nabla^2 U^{y}_T(t,z)[v,v] \geq \alpha \coth(\alpha(T-t)).
\ees

 \end{lemma}
 \begin{proof}
 Using Lemma \ref{genhtrans2} with the choices $U^1:=\frac{\alpha}{2}|z|^2$ and $U^2:=U$, we have
 \bes
U^{2,y}_{T}(t,z) - U^{1,y}_T(t,z) =U^{y}_{T}(t,z) - U^{1,y}_T(t,z)
 =   - \log  \underbrace{E_{P^{1,z,y}_{t,T}}\left( \exp(-\int_{t}^{T}  \mathscr{U}(\omega_s)-\frac{\alpha^2}{2}|\omega_s|^2 ds) \right)}_{(*)}  
 \ees
 where $P^{1,z,y}_{t,T}$ is the Ornstein-Uhlenbeck bridge (i.e. $U(z) = \frac{\alpha}{2}|z|^2 $).
 
Defining $\bar{\scrU}(z):=\mathscr{U}(z)-\frac{\alpha^2}{2}|z|^2$  we can rewrite $(*)$ by discretizing time as
 \be\label{eq:discretise}
\lim_{N \uparrow + \infty} \frac{\int_{\R^{d(N-1)}} \exp\Big( \sum_{i=0}^{n-1}-\bar{\scrU}(x_i)(T-t)/N + \log(p_{(T-t)/N}(x_i,x_{i+1})\Big) dx_1..dx_{n-1} }{\int_{\R^{d(N-1)}} \exp\Big( \sum_{i=0}^{n-1} + \log(p_{(T-t)/N}(x_i,x_{i+1})\Big) dx_1..dx_{n-1} }
 \ee
where $p_{(T-t)/N}(.,.)$ is the transition density of the (non pinned) Ornstein Uhlenbeck process and we adopted the convention that $x_0=z,x_n=y$. The hypothesis \eqref{eq:convass} tells that $\bar{\scrU}$ is a convex function. Moreover, since the Ornstein Uhlenbeck process is a Gaussian process, the function $\sum_{i=0}^{n-1} \log(p_{(T-t)/N})(x_i,x_{i+1})$ is, up to an affine transformation, a negative quadratic form in the variables $z=x_0,x_1...,x_N=y$ . 
These observations entitle  us to apply Pr\'ekopa-Leindler inequality in one of its quantitative versions, see e.g. \cite[Th. 4.3]{BrLieb76} or \cite[Th.13.13]{Simon2011} to conclude that for any $N$ the quantity in \eqref{eq:discretise}, viewed as function of $y$ and $z$, is log-concave. Passing to the limit as $N \rightarrow +\infty$ we obtain that 
 $(*)$ is log-concave as a function of $z$ (and $y$). Therefore for any $t\in [-T,T]$ and $z,v \in \RD$ such that $|v|=1$ we have:
 \bes
 \nabla^2 U^{y}_T(t,z)[v,v] \geq \nabla^2 U^{1,y}_T(t,z)[v,v] =\alpha \coth(\alpha(T-t)),
 \ees
where, since $U^1 = \frac{\alpha}{2} |z|^2$, we used the explicit calculations of point (i) of Lemma \ref{lemma:OUbridge} to derive the last equality .
\end{proof}
The third Lemma is about time reversal.
\begin{lemma}\label{lemma:timerev}
Fix $T>0$, and let $U$ be time-homogeneous. Denote by $\omega^*$ the time reversal of $\omega$, i.e. $(\omega^*_t)_{-T \leq t \leq T} = (\omega_{-t})_{ -T \leq t \leq T}$.
We have the following equality in distribution:
\bes
\omega^* \# P^{x,y}_{-T,T} = P^{y,x}_{-T,T}.
\ees
\end{lemma}
\begin{proof}
First, observe that if $U$ is time-homogeneous, then so is $\mathscr{U}$. Using the Girsanov Theorem as in \cite{LeKre93}:
\bes
\frac{d P^{x,y}_{-T,T}}{ d W^{x,y}_{-T,T}} = \frac{\exp( - \int_{-T}^{T} \scrU(\omega_s) ds)}{ E_{W^{x,y}_{-T,T}} ( \exp( -\int_{-T}^{T} \mathscr{U}(\omega_s) ds ) ) }
\ees
 It is easy to see that the conclusion holds for the Brownian bridge, i.e.
\bes
\omega^* \# W^{x,y}_{-T,T}= W^{y,x}_{-T,T}
\ees
But then, observing that $ \int_{-T}^{T}\mathscr{U}(\omega_s)ds = \int_{-T}^{T}\mathscr{U}(\omega^*_s)ds $, we have, for any test function $G$:
\beas
E_{P^{x,y}_{-T,T}} ( G(\omega^*) ) &=&   \frac{1}{ E_{W^{x,y}_{-T,T}} ( \exp( -\int_{-T}^{T} \mathscr{U}(\omega_s) ds ) ) }E_{W^{x,y}_{-T,T}} \left( \exp( - \int_{-T}^{T} \mathscr{U}(\omega_s) ds )  G(\omega^*) \right) \\
&=&\frac{1}{ E_{W^{x,y}_{-T,T}} ( \exp( -\int_{-T}^{T} \mathscr{U}(\omega^*_s) ds ) ) }E_{W^{x,y}_{-T,T}} \left( \exp( - \int_{-T}^{T} \mathscr{U}(\omega^*_s) ds )  G(\omega^*) \right) \\
&=&\frac{1}{ E_{W^{y,x}_{-T,T}} ( \exp( -\int_{-T}^{T} \mathscr{U}(\omega_s) ds ) ) }E_{W^{y,x}_{-T,T}} \left( \exp( - \int_{-T}^{T} \mathscr{U}(\omega_s) ds )  G(\omega) \right)  \\
&=& E_{P^{y,x}_{-T,T}} ( G(\omega) ) 
\eeas
from which the conclusion follows, as $G$ is arbitrary.
\end{proof}

We can now prove Theorem \ref{thm:grest}. In the proof, we call a coupling of two \textit{probability measures} $P^i,i=1,2$ any coupling $(X^1,X^2)$ of two random variables whose image measure are $P^1$  and $P^2$ respectively.
\begin{proof}

$(ii) \Rightarrow (i)$
Let $v \in \RD, |v| = 1$ and define $f(z)= v \cdot z$. Moreover define for $i=1,..,d$:
\bes
\varphi^i(t) = \partial_{x_i} E_{\LBR} f(\omega_t), \quad \varphi(t) = ( \varphi^1,...,\varphi^d)(t).
\ees

Using the fact that $|\varphi(-T) |=| \nabla_x f(x)| = |v|=1 $, the gradient estimate \eqref{eq:grest} can be rewritten as:
\bes
|\varphi(t) |- |\varphi(-T) | \leq \frac{\sinh(\alpha(T-t))}{\sinh(2 \alpha T)} - 1
\ees
Differentiating the above inequality at $t=-T$ we get, using again $|\varphi(-T) |=1$ and $\varphi(-T) = v$:
\be\label{eq:shorttime1}
v \cdot \partial_t \varphi(-T) \leq -\alpha \coth(2 \alpha T)
\ee
Exchanging space and time derivatives:
\beas
 \partial_t \varphi^i(-T) &=& \partial_{x_i} \partial_t E_{\LBR} f(\omega_t) \\
 &=& \partial_{x_i} \Big( \frac{1}{2} \Delta f(x) - \nabla U^{y}_T(-T,x) \cdot \nabla f(x)\Big)\\
 &=& - \sum_{j=1}^d v^j \partial_{x_i} \partial_{x_j} U^{y}_T(-T,x)  
\eeas
Plugging this into \eqref{eq:shorttime1} we obtain
\bes
- \nabla^2 U^{y}_T(-T,x)[v,v] \leq - \alpha \coth(2 \alpha T)
\ees
 Using the representation for $U^y_T$ given at \eqref{eq:driftreprs2} with the choices $U^1= \frac{\alpha}{2}|z|^2, U^2=U$, the inequality above together with point (i) of Lemma \ref{lemma:OUbridge} imply that
\be\label{eq:tildeU}
\nabla^2 \log \psi(-T,x)[v,v] \leq 0, \ \text{where} \quad \psi(-T,z)=  E_{P^{1,x,y}_{-T,T}}\exp \Big(-\int_{-T}^T \tilde{\scrU}(\omega_s) ds  \Big), \tilde{\scrU}:= \scrU(z) - \frac{\alpha^2}{2}|z|^2 + d \alpha ,
\ee
and we recall that with the choices made $P^{1,x,y}_{-T,T}$ is the Ornstein Uhlenbeck bridge. \eqref{eq:tildeU} is equivalent to:
\be\label{eq:shorttime2}
\nabla^2 \psi(-T,x)[v,v] \psi(-T,x) - (\nabla \psi(-T,x)\cdot v)^2 \leq 0
\ee
Using the equality in distribution \eqref{eq:OUbridgetranslation} from point (ii) of Lemma \ref{lemma:OUbridge} and the definition of $\psi$, we obtain
\bes
\psi(-T,x+\varepsilon v ) = E_{P^{1,x+ \varepsilon v ,y}_{-T,T}}\exp\Big(-\int_{-T}^T \tilde{\scrU}(\omega_t) dt \Big) = E_{P^{1,x,y}_{-T,T}}\exp\Big(-\int_{-T}^T \tilde{\scrU}(\omega_t + \varepsilon v \xi_T(t)) dt \Big)
\ees
We  can use this last identity to compute the spatial derivatives of $\psi$. To ease notation, we define
\bes
M_T:= \exp\Big(-\int_{-T}^T \tilde{\scrU}(\omega_t)dt \Big)
\ees
We then obtain, after some computations:
\beas
\nabla^2 \psi(-T,x)[v,v] \psi(-T,x) &=&- E_{P^{1,x,y}_{-T,T}}\Big( M_T \int_{-T}^T \nabla^2 \tilde{\scrU} (\omega_t)[v,v] \xi^2_T(t) dt\Big) \psi(-T,x)\\
 &+&   E_{P^{1,x,y}_{-T,T}}\Big( M_T \Big( \int_{-T}^T \nabla\tilde{\scrU} (\omega_t) \cdot v  \, \xi_T(t) dt \Big)^2  \Big) \psi(-T,x) \\
\eeas
and
\bes
(\nabla \psi(-T,x)\cdot v )^2 = \left[ E_{P^{1,x,y}_{-T,T}}\Big( M_T\int_{-T}^T \nabla \tilde{\scrU}  (\omega_t)\cdot v\, \xi_T(t) dt\Big) \right]^2
\ees

Observing that $ \psi(-T,x)=E_{P^{1,x,y}_{-T,T}}(M_T) $ we can use  Cauchy-Schwarz inequality to conclude that
\bes
 E_{P^{1,x,y}_{-T,T}}\Big( M_T \big( \int_{-T}^T \nabla\tilde{\scrU} (\omega_t) \cdot v  \, \xi_T(t) dt \big)^2  \Big) \psi(-T,x)
- \left[E_{P^{1,x,y}_{-T,T}}\Big( M_T\int_{-T}^T \nabla \tilde{\scrU}  (\omega_t)\cdot v\, \xi_T(t) dt\Big)\right]^2 \geq 0.
\ees
Using this last fact and that $\psi(-T,x)>0$, \eqref{eq:shorttime2} implies that for all $T > 0$:
\be\label{eq:shorttime3}
E_{P^{1,x,y}_{-T,T}}\Big( M_T \, \int_{-T}^T \nabla^2 \tilde{\scrU}(\omega_t)[v,v] \xi^2_T(t) dt \Big) \geq 0
\ee
Recalling the definition of $S_T$ at point (iii) of Lemma \ref{lemma:OUbridge} we can rewrite:
\beas
\frac{1}{T} E_{P^{1,x,y}_{-T,T}}\Big( M_T \, \int_{-T}^T \nabla^2 \tilde{\scrU}[v,v] (\omega_t) \xi^2_T(t) dt \Big)\\
= E_{P^{1,x,y}_{-T,T}}\Big( (M_{1})^{T}(S_T\omega)\,  \int_{-1}^1 \nabla^2 \tilde{\scrU} ((S_T\omega)_{t}) [v,v]\xi^2_T(Tt) dt \Big)
\eeas
Setting $y=x$, equation \eqref{eq:timechange} from point (iii) of Lemma \ref{lemma:OUbridge} and the fact that $\lim_{T \downarrow 0}\xi_{T}(Tt)= \frac{(1-t)}{2}$ we have that
\bes
 \lim_{T \rightarrow 0} \frac{1}{T}E_{P^{1,x,x}_{-T,T}}\Big( M_T \, \int_{-T}^T \nabla^2 \tilde{\scrU}(\omega_t)[v,v] \xi_T(t) dt \Big) = \frac{2}{3}\nabla^2 \tilde{\scrU}(x)[v,v],
\ees
and because of \eqref{eq:shorttime3} we can conclude that $\nabla^2 \tilde{\scrU}(x)[v,v] \geq 0$.
Recalling the definition of $\tilde{\scrU}$ given at \eqref{eq:tildeU} the conclusion follows, since $x$ and $v$ arbitrarily chosen.

$(i) \Rightarrow (iii)$
We proceed as follows: we first show that there exist a coupling $(X^1,X^2)$ of $P^{x_1,y_1}_{T,T}$ and $P^{x_2,y_1}_{T,T}$ such that 
\be\label{eq:couplinginitialpoints}
|X^1_t -  X^2_t| \leq |x_2-x_1| \frac{ \sinh( \alpha(T-t) )  }{\sinh( 2 \alpha T)}
\ee
 holds almost surely for any $-T \leq t \leq T$.
Next we show that there exist a coupling $(Y^1,Y^2)$ of $P^{x_2,y_1}_{T,T}$ and $P^{x_2,y_2}_{T,T}$ such that
\be\label{eq:couplingfinalpoints}
|Y^1_t -  Y^2_t| \leq |y_2-y_1| \frac{ \sinh( \alpha(T+t) )  }{\sinh( 2 \alpha T)}
\ee
 holds almost surely for any $-T \leq t \leq T$.
 From the existence of these two couplings, it is then easy to deduce the conclusion.

Let us proceed with the first step. Consider some Brownian  motion $(B_{t})_{-T \leq t \leq T}$  and for $i=1,2$ let $X^{i}$ be a strong solution for
\bes
dX^i_{t} = -\nabla U^{y^1}_{T}(t,X^i_t)dt + dB_t, \quad X^i_{-T} =x^i 
\ees
We have, using standard computations, and the convexity estimate from Lemma \ref{lem:conc}:
\beas
&{}&|X^1_t -  X^2_t|^2 - |X^1_s-X^2_s|^2 \\
&=& 2\int_{s}^t (X^1_r -  X^2_r) \cdot ( \nabla U^{y_2,T}(X^2_r) - \nabla U^{y_1,T}( X^1_r) ) dr \\
&\leq & -2 \alpha \int_{s}^t \coth( \alpha(T-r) ) |X^1_r -  X^2_r|^2  dr
\eeas
Therefore the function $\varphi(t):= |X^1_t -  X^2_t|^2 $ satisfies the differential inequality:
\bes
\dot{\varphi}_{t} \leq -2 \alpha \coth(\alpha (T-t))  \varphi_t  \quad \varphi_{-T}=|x_2-x_1|^2 
\ees
Integrating this differential inequality and taking the square root we obtain \eqref{eq:couplinginitialpoints}.
In the same way, one shows the existence of a coupling $(Z^1,Z^2)$ of $P^{y_1,x_2}_{-T,T}$ and $P^{y_2,x_2}_{-T,T}$ such that 
\bes
|Z^1_t -  Z^2_t| \leq |y_2-y_1| \frac{ \sinh( \alpha(T-t) )  }{\sinh( 2 \alpha T)}
\ees
hold almost surely for any $-T \leq t \leq T$. But then thanks to the time reversal relation of Lemma \eqref{lemma:timerev} we have that $(Y^1,Y^2):=((Z^1)^*,(Z_2)^*)$ is a coupling of $P^{x_2,y_1}_{-T,T}$ and $P^{x_2,y_2}_{-T,T}$ such that \eqref{eq:couplingfinalpoints} holds.

$(iii)\Rightarrow (ii)$ The fact that a coupling implies a gradient estimate is well known, and can be proven easily by considering expectations in \eqref{eq:couplinginitialpoints}
,\eqref{eq:couplingfinalpoints} and diffeentiating in the $x$ and $y$ variable respectively.
\end{proof}
\subsection{Proof of Theorem \ref{thm:logsob}}
We first prove the following characterization of the Ornstein Uhlenbeck bridge based on the results of \cite{RT05}.
\begin{lemma}\label{lemma:isonormal}
Let $\nabla \scrU = \alpha^{2} z $. Then $P^{0,0}_{-T,T}$ is a centred Gaussian process on $\Omega^{0,0}_{-T,T}$ such that for all $g,h \in H$:
\bes
E_{P^{0,0}_{-T,T}}\Big(\int_{-T}^T h_t \cdot d \omega_t \int_{-T}^{T} g_t \cdot d \omega_t \Big):= \langle h,g \rangle_{\alpha}
\ees
\end{lemma}
\begin{proof}
Theorem 4.1 in \cite{RT05} is a characterization of $P^{0,0}_{-T,T}$ implying that for any  smooth test function $F$ and any $\varphi \in H$ satisfying the loop condition $\int_{-T}^T \varphi_t dt =0$ the formula
\begin{equation*}
E_{P^{0,0}_{-T,T}} \Big( \langle DF, \varphi \rangle_{H} \Big) =E_{P^{0,0}_{-T,T}} \left( F \int_{-T}^{T} \Big( \varphi_t   - \alpha^2 \int_{-T}^t\int_{-T}^s \varphi_r dr\, ds \Big) \,\cdot d\omega_t  \right)
\end{equation*}
holds, where $DF$ is the Malliavin derivative. Fix $h,g$ and let $\varphi$ be defined by \eqref{eq:ODEforh}. We observe that $\varphi$ satisfies by construction the loop condition. If we choose $F = \int_{-T}^T g_t \cdot d \omega_t$, then $\langle DF,\varphi \rangle_{H} = \langle \varphi,g \rangle_H$. Using \eqref{eq:ODEforh} and taking advantage of the fact that $\int_{-T}^T c \cdot d \omega_t =0 $ almost surely for any constant $c$ since $P^{0,0}_{-T,T}(\omega_{T}=\omega_{-T}=0)=1$:
\bes
\langle \varphi,g \rangle_{H} = E_{P^{0,0}_{-T,T}}\Big(\int_{-T}^{T} g_t \cdot d \omega_t \int_{-T}^{T} h_t \cdot d \omega_t\Big) 
\ees
which is the desired conclusion.
\end{proof}
We can get back to the proof of the Theorem.
\begin{proof}
 Let $F=f\Big(\int_{-T}^{T} h^1_t \cdot d\omega_t,..,\int_{-T}^{T} h^n_t \cdot d\omega_t \Big)$. W.l.o.g. we can assume that $\langle h^i,h^j\rangle_{\alpha} = \delta_{ij}$. This means,thanks to Lemma \ref{lemma:isonormal}, that under the bridge of the Ornstein-Uhelnbeck process with mean reversion $\alpha$ (which we call ${}^{\alpha}P^{0,0}_{-T,T}$) the  law  ${}^{\alpha}\mu:= (\int_{-T}^{T} h^1_t \cdot d\omega_t,..,\int_{-T}^{T} h^n_t \cdot d\omega_t) \# {}^{\alpha}P^{0,0}_{-T,T}  $ is a centred Gaussian with identity covariance matrix. If we define 
$$\mu := (\int_{-T}^T h^1_t \cdot d\omega_t,..,\int_{-T}^T h^n_t \cdot d \omega_t) \# P^{0,0}_{-T,T}  $$ we have that
\bes
\frac{d \mu}{d {}^{\alpha} \mu} (z_1,..,z_n)\propto E_{{}^{\alpha}P^{0,0}_{-T,T}} \Big[ \exp\Big(-\int_{-T}^{T} \scrU(s,\omega_s) dt\Big) \Big| \int_{-T}^T h^1_t \cdot d\omega_t = z_1,.., \int_{-T}^{T} h^n_t \cdot d\omega_t = z_n \Big]
\ees

Using again the quantitative version of Prekopa Leindler inequality (see \cite[Th.13.3]{Simon2011}) as in the proof of Lemma \ref{lem:conc}  one has that $\frac{d \mu}{d {}^{\alpha} \mu}$ is log concave. Therefore $\mu$ has a log concave density against the standard normal distribution. Using well known results (see e.g. \cite[Thm 5.2]{Led01}) we obtain that  $\mu$ satisfies the LSI on $\R^n$ with constant $2$. Therefore
\be\label{eq:logsob}
E_{\mu}(f \log f)- E_{\mu}(f)\log E_{\mu}(f) \leq 2 E_{\mu}\Big(\frac{1}{f} \sum_{j=1}^n |\partial_i f|^2 \Big),
\ee
which means
\bes
E_{P^{0,0}_{-T,T}}(F \log F) - E_{P^{0,0}_{-T,T}}(F)\log E_{P^{0,0}_{-T,T}}(F)\leq 2 E_{P^{0,0}_{-T,T}}\Big(\frac{1}{F} \sum_{i=1}^n |\partial_i f|^2\Big(\int_{-T}^T h^1_t \cdot d\omega_t ,..,\int_{-T}^T h^n_t \cdot d \omega_t\Big) \Big)
\ees
Since $ \langle h^i,h^j\rangle_{\alpha} = \delta_{ij}$,
$$\sum_{i=1}^n |\partial_i f|^2\Big(\int_{-T}^T h^1_t \cdot d\omega_t ,..,\int_{-T}^T h^n_t \cdot d \omega_t \Big) = \langle DF,DF \rangle^2_{\alpha}$$
from which the conclusion follows.
\end{proof}
Here is the proof of Corollary \ref{cor:conc}.
\begin{proof}
Fix $t>0$ and let   $h^i_s:= \mathbf{1}_{[-T,t]}(s) \mathbf{e}_i$ for $i=1,..,d$ , where $\mathbf{e}_i$ is the i-th vector of the canonical basis of $\RD$. Then we have that $f\Big(\int_{-T}^T h^1_t d \omega_t,..,\int_{-T}^T h^d_t d \omega_t) = f(\omega_t)$. The inequality, applied to the simple functional $F(\omega)=f(\omega_t) = f\Big(\int_{-T}^T h^1_t \cdot  d \omega_t,..,\int_{-T}^T h^d_t \cdot d \omega_t)$ tells that the entropy of $F$ is bounded by:
\bes
2 \sum_{i=1}^{d} (\partial_i f(\omega_t))^2   \langle h^i, \varphi^i \rangle_{\alpha}
\ees
  where any $i=1,..,d$ $\varphi^i$ is the solution of the equation \eqref{eq:ODEforh}. It can be verified with a direct computation that 
  \bes
  \varphi^i_s= \Big[\frac{\sinh(\alpha(T-t))}{\sinh(2\alpha T)}\cosh(\alpha(T+s)) - \mathbf{1}_{[t,T]}(s) \cosh(\alpha (s-t)) \Big] \mathbf{e}_i.
  \ees
  This means that, independently from $i$ we have:
  \bes
  \langle h^i, \varphi^i \rangle_{\alpha}= \frac{\sinh(\alpha(T-t))}{\sinh(2\alpha T)}\int_{-T}^t\cosh(\alpha(T+s)) ds=\frac{\sinh(\alpha(T-t))\sinh(\alpha(T+t))}{\alpha \sinh(2\alpha T)}
  \ees
  Summing up, we have proven that $\omega_t \# P^{0,0}_{-T,T}$ satisfies the Logarithmic Sobolev inequality on $\RD$ with coefficient $2 \frac{\sinh(\alpha(T-t))\sinh(\alpha(T+t))}{\alpha \sinh(2\alpha T)}$. The fact that the LSI implies the desired concentration bound is well known, see e.g. \cite[Th. 5.3]{Led01}.
\end{proof}


\subsection{Proof of Theorem \ref{thm:comparisonprinciple} and Corollary \ref{cor:comparison}  }
  First, let us observe that if Corollary \ref{cor:comparison} is proven, Theorem \ref{thm:comparisonprinciple} follows by mixing over the bridge endpoints. Therefore, let us focus on the proof of the Corollary.

  \begin{proof}

According to Lemma \ref{genhtrans2}, the drift of $ P^{2,x,y}_{-T,T}$ is:
\bes
-\partial_z U^{1,y,T}(t,z) + \partial_{z} \log E_{P^{1,z,y}_{t,T}}\Big( \exp\Big(-\int_{t}^T [\scrU^2-\scrU^1](s,\omega_s)ds\Big) \Big)
\ees	
Using this representation we observe if we could show that for any $t$ the function \\ $z \mapsto   E_{P^{1,z,y}_{t,T}}\Big( \exp \Big(\int_{t}^T [\scrU^1-\scrU^2](s,\omega_s)ds\Big) \Big)$ is increasing, we would be done, thanks to the standard comparison principle for SDEs such as \cite[Ch.IX, Thm 3.7]{revuz2013continuous}.
To do this, we need to first introduce some notation.
For two paths $\omega,\tilde{\omega} \in \Omega_{t,T}$ we write $\omega \preceq \tilde{\omega} $ if and only if $\omega_s \leq  \tilde{\omega}_s$ for all $t \leq s \leq T$.
We say that a functional $F$ is \emph{increasing} if
$F(\omega) \leq F(\tilde{\omega})$ whenever $\omega \preceq \tilde{\omega}$.
 Take now $\tilde{z}\geq z$. Since $P^{1,z,y}_{t,T}$ and $P^{1,\tilde{z},y}_{t,T}$ are the laws of two one dimensional diffusion processes with the same generator and $\tilde{z}\geq z$ the comparison principle
for SDE tells that we can find a probability space $(\tilde{\Omega},\tilde{F},\tilde{P})$ and two maps $X^z,X^{\tilde{z}}$ defined on it such that
\bes X^z \# \tilde{P} = P^{1,z,y}_{t,T}, X^{\tilde{z}} \# \tilde{P} = P^{1,\tilde{z},y}_{t,T}, \quad \text{and} \, X^z \preceq X^{\tilde{z}} \, \tilde{P}-\text{almost surely}. \ees
Take now any increasing functional $F$. We have
\bes
E_{P^{1,z,y}_{t,T}}(F(\omega))= E_{\tilde{P}}(F(X^z)) \leq E_{\tilde{P}}(F(X^{\tilde{z}}) )= E_{P^{1,\tilde{z},y}_{t,T}}(F(\omega))
\ees
so that $z \mapsto E_{P^{1,z,y}_{t,T}}(F)$ is an increasing function. The hypothesis of the Theorem makes sure that $ \omega \mapsto \exp\Big(\int_{t}^T [\scrU^1-\scrU^2](s,\omega_s)ds\Big)$ is an increasing functional.  Therefore, for any $t$, the function  $ z \mapsto E_{P^{1,z,y}_{t,T}}\Big(\exp\Big(\int_{t}^T \scrU^1-\scrU^2(s,\omega_s)ds\Big)\Big)$ is increasing, which is the desired conclusion.
 \end{proof}

 \subsection{Proof of Theorem \ref{thm:Stein}}
For $f$ twice Frech\'et differentiable w.r.t. to $\|.\|_{\infty}$ we write $Df(\omega)[\omega']$ for the directional derivative of $f$ at $\omega$ in direction $\omega'$. Similarly, we write $D^2f(\omega)[\omega',\omega'']$ for the second derivative. If $\omega'=\omega''$ we simply write $D^2f(\omega)[\omega'].$ The proof relies on the following Lemma, where we construct and splve a Stein equation for the Brownian bridge, following closely the ideas of \cite{barbour1990stein} for the Brownian motion.
\begin{lemma}
Consider the semigroup $(S_u)_{u \geq 0}$ on $\Omega^{x,y}_{-T,T}$ defined for any $f$ bounded and continuous by:
\be\label{SPDEsemigroup}
S_uf(\omega) = \int_{\Omega_{-T,T}^{0,0}} f(\omega e^{-u} + \sigma(u) \tomega +(1-e^{-u})\varphi ) W^{0,0}_{-T,T}(d \tomega) 
\ee
 where $\sigma(u)= (1-e^{-2u})^{1/2}$ and $\varphi_{t} = \frac{(T-t)}{2T}x + \frac{(T+t)}{2T}y$.
The following holds:

\begin{enumerate}[(i)]

\item For any twice Frech\'et differentiable $f$ the generator $\scrA$ of $(S_u)_{u \geq 0}$ is 
\be\label{def:SPDEgenbrbr}
\scrA f(\omega) = -Df(\omega)[\omega - \varphi]+  \int_{\Omega_{-T,T}^{0,0}} D^2f(\omega)[\tomega]\, W^{0,0}_{-T,T}(d \tomega) .
\ee

\item $\LBBR$ satisfies the integration by parts formula:
\be\label{IPBFOUsemigroup}
E_{\LBBR}(f \scrA g) = E_{\LBBR}( g \scrA f) 
\ee
for all smooth $f,g$. In particular, $E_{\LBBR}(\scrA f) =0$.

\item Let $f$ be 1-Lipschitz and $E_{\LBBR}(f)=0$. The solution to the equation
 $$\scrA g(\omega) = f(\omega)  $$
is given by
\be\label{solutionsteineq}
g(\omega) = - \int_{0}^{+\infty}S_u f(\omega)du
\ee
and $g$ is 1-Lipschitz as well.

\item $P^{x,y}_{-T,T}$ satisfies the formula

\be\label{eq:Langevinbridgegen1}
E_{P^{x,y}_{-T,T}} \Big(\scrA_{\scrU} f \Big) = 0
\ee

where
\be\label{eq:Langevinbridgegen2}
\scrA_{\scrU} f = \scrA f -  \int_{\Omega^{0,0}_{-T,T}} \int_{-T}^{T}\big[ Df(\omega)[\tomega] \nabla \scrU (s,\omega_s) \cdot \tomega_s \big] \,\, ds \,  W^{0,0}_{-T,T}(d \tomega)
\ee
\end{enumerate}

\end{lemma}
\begin{proof}
To ease the notation, we only prove the Lemma for $x=y=0$. The case of general $x,y$ can be proven along the same lines with minimal changes.
The proof of (i) amounts to the observation that the discussion preceding Theorem 1 in \cite{barbour1990stein} is valid for any Gaussian process and not just the Brownian motion: for this reason, we skip the details here.
To prove $(ii)$ we show that 
\bes
E_{W^{0,0}_{-T,T}}(f S_u g ) = E_{W^{0,0}_{-T,T}}(S_u f  g ) 
\ees
from which the conclusion follows.
Observe that we can rewrite:
\bes
E_{W^{0,0}_{-T,T}}(f S_u g ) = \int_{\Omega^{0,0}_{-T,T} \times \Omega^{0,0}_{-T,T}}f(\omega) g(e^{-u} \omega + \sigma(u) \tomega ) W^{0,0}_{-T,T} \otimes W^{0,0}_{-T,T}(d \omega \, d \tomega).
\ees
As it can be verified with a direct computation, image measure of $W^{0,0}_{-T,T} \otimes W^{0,0}_{-T,T} $
under the mapping 
\bes
(\omega, \tomega) \mapsto (\sigma(u) \omega - e^{-u} \tomega , e^{-u} \omega +  \sigma(u) \tomega):=(\omega_1,\tomega_1)
\ees
is again $W^{0,0}_{-T,T} \otimes W^{0,0}_{-T,T}$. By definition we have $ \omega = e^{-u} \tomega_1 + \sigma(u) \omega_1 $, so that
\beas
&{}&\int f(\omega) g(\exp(-u) \omega + \sigma(u) \tomega ) W^{0,0}_{-T,T}\otimes W^{0,0}_{-T,T}(d \omega \, d \tomega)\\
&=&\int f(e^{-u}\tomega_1 + \sigma(u) \omega_1) g(\tomega_1 ) W^{0,0}_{-T,T} \otimes W^{0,0}_{-T,T}(d \omega_1 \, d \tomega_1)\\
&=& E_{W^{0,0}_{-T,T}}((S_u f)g)
\eeas
which is the desired result. Let us prove (iii): the proof that $g$ is well defined and solves the Stein equation is a standard computation using the fact that $ W^{0,0}_{-T,T}$ is the reversible measure for $\scrA$. The fact that $g$ is 1-Lipschitz also follows from the definition of $g$ and the explicit expression for $S_u$ given at \eqref{SPDEsemigroup}.

To prove (iv) we recall the Girsanov Theorem for bridges:  
\bes
\frac{dP^{0,0}_{-T,T}}{dW^{0,0}_{-T,T}} =  \frac{1}{Z}\exp(- \int_{-T}^{T} \scrU(s,\omega_s)ds ):=M
\ees
Using the product rule and the integration by parts for the invariant measure we have for all $f$:
\beas 
0 &=& E_{W^{0,0}_{-T,T}}\Big(\scrA(fM)\Big) = E_{W^{0,0}_{-T,T}}\Big(M \scrA f + f \scrA M + 2 \int_{\Omega^{0,0}_{-T,T}} Df[\tomega]DM[\tomega] W^{0,0}_{-T,T}(d \tomega) \Big) \\
&=& 2 E_{W^{0,0}_{-T,T}}\Big( M \{ \scrA f+  \int_{\Omega^{0,0}_{-T,T}} Df[\tomega]D \log M[\tomega]  W^{0,0}_{-T,T}(d \tomega) \} \Big)\\
&=& 2 E_{P^{0,0}_{-T,T}}\Big( \scrA f+   \int_{\Omega^{0,0}_{-T,T}} Df[\tomega]D\log M[\tomega]  W^{0,0}_{-T,T}(d \tomega) \Big)
\eeas
It can be checked with a direct computation that:
\bes
D \log M(\omega)[\tomega] = -\int_{-T}^T \nabla \scrU (s,\omega_s) \cdot \tomega_s ds
\ees
which gives the conclusion.

\end{proof}
At this point, the proof of Theorem \ref{thm:Stein} is almost straightforward.

\begin{proof}
Let $f$ be 1-Lipschitz and $x,y$ be fixed. Then we have, using \eqref{solutionsteineq}, \eqref{eq:Langevinbridgegen1} and \eqref{eq:Langevinbridgegen2}:
\beas
&{}&|E_{\LBR}(f)-E_{W^{x,y}_{-T,T}}(f)| =| E_{\LBR}(\scrA g)|= | E_{\LBR}(\scrA g - \scrA_{\scrU} g )| \\
&=& \int_{\OMB \times \Omega^{0,0}_{-T,T}}Dg(\omega)[\tomega] \int_{-T}^{T} \nabla\scrU(s,\omega_s) \cdot \tomega_s ds \, \LBR(d \omega) \otimes W^{0,0}_{-T,T}(d \tilde{\omega}) 
\eeas
Since $g$ is 1-Lipschitz we have: $ |Dg(\omega)[\tomega]| \leq \|\tomega \|_{\infty}$. Moreover, using Cauchy-Schwartz inequality we have $|\nabla\scrU(s,\omega_s)\cdot \tomega_s |\leq \| \nabla \scrU \|_{\infty}|\tomega_s|$, Therefore we get:
\beas
|E_{\LBR}(f)-E_{W^{x,y}_{-T,T}}(f)| &\leq & \| \nabla \scrU \|_{\infty} E_{W^{0,0}_{-T,T}}\Big(\| \tomega \|_{\infty} \int_{-T}^{T} |\tomega_s|ds \Big) \\
&= &\| \nabla \scrU \|_{\infty} T E_{W^{0,0}_{-T,T}}\Big(\sup_{s \in [-1,1]  }|\tomega_{Ts}| \int_{-1}^{1} |\tomega_{Ts}|ds \Big)
\eeas
Define now $Z_s(\tilde{\omega}) \frac{1}{\sqrt{T}}S_T(\tilde{\omega})$ (see point $(iii)$ of lemma \ref{lemma:OUbridge} for the definition of $S_t$). It is 
easy to check that $ Z_T \# W^{0,0}_{-T,T} = W^{0,0}_{-1,1}$. Using this distributional equality in the above equation proves \eqref{eq:steinbridgebound}. To pass to a general $\mu \in \cP(\RD \times \RD)$ we observe that, because of the fact that $(\omega_{-T},\omega_T) \# \PMU = (\omega_{-T},\omega_T) \# W^{\mu}_{-T,T}=\mu$ we have for any $1$-Lipschitz function $f$:
\beas
|E_{\PMU}(f)- E_{W^{\mu}_{-T,T}}(f) |= |\int_{\RD \times \RD}E_{\LBR}(f)-E_{\LBBR}(f) \mu(dx dy)| \\
\leq |\int_{\RD \times \RD}W_1(\LBR,\LBBR)\mu(dx dy)| 
\eeas
Using the bound on the Wasserstein distance for the bridge case, which does not depend on $x,y$, \eqref{eq:steinreciprocalbound} follows.
\end{proof}

\subsection{Proof of Theorem \ref{theorem:invmeas}}
We begin by proving point (i).
\begin{proof}
First, let us observe that we can reduce ourselves to show that $\omega_0 \# P^{x,y}_{-T,T} \rightarrow m$ for all $x,y \in \RD$.
Next, we claim that if we can find a convex potential $V$ such that 
$\nabla \mathscr{V}(z) = \nabla \mathscr{U}(z)$ everywhere on $\RD$, the existence of $m$ is proven.  To see this, assume that such $V$ exists and denote by $Q^{x,y}_{-T,T}$ the bridge of the Langevin dynamics associated with $V$. It is well known (see e.g. \cite{Kre88},\cite{Cl91}) that:
\bes 
\forall \, x,y \in \RD, T>0 \quad Q^{x,y}_{-T,T} = P^{x,y}_{-T,T} .\ees

  If we denote by $q_t(x,y)$ the transition  kernel of $\partial_t- \frac{1}{2} \Delta + \nabla V \cdot \nabla $, we have that the density of  $\omega_0 \# P^{x,y}_{-T,T}$ ($=\omega_0 \# Q^{x,y}_{-T,T}$) w.r.t. to the Lebesgue measure is:
\bes
 \varphi^{x,y}_{-T,T}(z)= \frac{q_{T}(x,z) q_{T}(z,y)}{q_{2T}(x,y)}
\ees
Since $V$ is convex, the corresponding Langevin dynamics converges to the invariant measure $m = \frac{1}{Z}\exp(-2V)  \lambda$. In particular
\bes
q_T(x,z) \rightarrow \frac{1}{Z}\exp(-2V(z)),  \quad q_T(z,y) \rightarrow \frac{1}{Z}\exp(-2V(y)),  
\quad q_T(z,y) \rightarrow \frac{1}{Z}\exp(-2V(y))  
\ees
and therefore $\varphi^{x,y}_{-T,T}(z) \rightarrow \frac{1}{Z}\exp(-2V(z))$, proving the existence of an invariant measure. 
Now we show that if $\mathscr{U}$ satisfies \eqref{eq:convass}, such potential $V$ can be found.
Consider the Schr\"odinger operator $- \frac{1}{2}\Delta  + \mathscr{U}$.
Since $\scrU$ is uniformly convex, $\lim_{|z| \rightarrow +\infty} \scrU(z) = +\infty$. This is known to be a sufficient condition to ensure the existence of a ground state $\psi$ for this operator. We call $k$ the principal eigenvalue.
 Since the ground state is everywhere positive, $V:= -\log \psi$ is well defined.
A simple manipulation of the equation $ - \frac{1}{2}\Delta \psi + \mathscr{U} \psi = k \psi $ shows that $V$ solves
\bes
\frac{1}{2}|\nabla V |^2 - \frac{1}{2}\Delta V + k =\mathscr{U} 
\ees
But this exactly means that $  \mathscr{V}+ k =\mathscr{U}$ and therefore $\nabla \mathscr{V} = \nabla \mathscr{U}$. Thanks to the convexity assumption \eqref{eq:convass} we are entitled to apply Theorem 6.1 of \cite{BrLieb76}. It tells that $\psi(z) = \exp(-1/2 \alpha |z|^2)\phi(z)$, where $\phi$ is log-concave. This immediately gives the desired conclusion about convexity of $V$.
\end{proof}

Let us introduce some notation for the proof of $(ii)$.
 
 We denote the typical element of $\RD \times \RD$ by $(x_1,x_2)$ and couplings of measures $(\delta,\rho)$ on $\RD$ by $\pi$.
 The set of all couplings is denoted $\Pi(\delta,\rho)$.
We denote $((\tilde{x}_1,\tilde{y}_1),(\tilde{x}_2,\tilde{y}_2))$ the typical element of $(\RD\times\RD) \times (\RD \times \RD)$. A coupling of two probability measures $\mu$ and $\nu$ on $\RD \times \RD$ is denoted $\tilde{\pi}$ and $\tilde{\Pi}(\mu,\nu)$ is the set of couplings. 

Recall the definition of $p$-Wasserstein distance of two measures on $\RD$, for $p \geq 1$ : 

\be\label{def:pwasser}
W_p(\delta,\rho ) = \inf_{\pi \in \Pi(\delta,\rho)}\left(\int |x_2-x_1|^p  \, d\pi\right)^{1/p}
\ee

The $p$-Wasserstein distance $\tilde{W}_p(\cdot,\cdot)$ for measures on $\RD \times \RD$ is defined in the same way.
We can now prove (ii).

\begin{proof}
Fix $p\geq 1$ and let $\tilde{\pi}$ be an optimal coupling for $\tilde{W}_p(\mu,\nu)$ . Then  define $\pi \in \Pi( \omega_0 \# P^{\mu}_{-T,T},\omega_0 \# P^{\nu}_{-T,T})$ via
\bes
\int_{\RD \times \RD} f(x_1,x_2) d \pi  = \int E(f(X^{\tilde{x}_1,\tilde{y}_1,T}_0,X^{\tilde{x}_2,\tilde{y}_2,T}_0)) d \tilde{\pi}
\ees
where $(X^{\tilde{x}_1,\tilde{y}_1,T},X^{\tilde{x}_2,\tilde{y}_2,T})$ is a coupling of $ P^{\tilde{x}_1,\tilde{y}_1}_{-T,T}$ and $P^{\tilde{x}_2,\tilde{y}_2}_{-T,T} $ such that
\bes
| X^{\tilde{x}_1,\tilde{y}_1,T}_t - X^{\tilde{x}_2,\tilde{y}_2,T}_t| \leq  \frac{\sinh(\alpha(T-t))}{\sinh(2\alpha T)} |\tilde{x}_2-\tx_1| + \frac{\sinh(\alpha (T+t) ) }{ \sinh(2\alpha T)} |\ty_2-\ty_1|  \quad \forall \, t \in [-T,T]
\ees
holds almost surely. Such a coupling can be constructed thanks to Theorem \ref{thm:grest}.
It is easy to see that $\pi \in \Pi(\omega_0 \# P^{\mu}_{-T,T},\omega_0 \# P^{\nu}_{-T,T})$.
But then:
\beas
W_p(\omega_0 \# P^{\mu}_{-T,T},\omega_0 \# P^{\nu}_{-T,T})^p&\leq & \int |x_2-x_1|^p d \pi \\
&=& \int E(|X^{\tilde{x}_2,\tilde{y}_2,T}_0-X^{\tilde{x}_1,\tilde{y}_1,T}_0|^p) d \tilde{\pi} \\
& \leq &
\frac{\sinh(\alpha T)   }{ \sinh(2\alpha T)}^p \int \left[ |\tx_2-\tx_1| +  |\ty_2 - \ty_1| \right]^p d \tilde{\pi} \\
&=&(2 \cosh(\alpha T))^{-p} \,2^{p/2}  \int  |(\tx_2,\ty_2)-(\tx_1,\ty_2)|^p d \tilde{\pi} \\
&=&(\sqrt{2} \cosh(\alpha T))^{-p} \, \,\tilde{W}_p(\mu,\nu)^p
\eeas
which proves \eqref{eq:contrest2}. To prove \eqref{eq:contrest} we first observe that $\hat{P}^m$ is well defined as $V$ is convex, (see \eqref{eq:invconvest}). By construction, $ \hat{P}^{m}_{-T,T}:=((\omega_{t})_{t \in [-T,T]}) \# \hat{P}^m \in \mathfrak{R}_T(U)$ because $\nabla \scrU = \nabla \scrV $. The conclusion  follows with application of the estimate \eqref{eq:contrest2} to $P^{\mu}_{-T,T}$ and $\hat{P}^{m}_{-T,T}$.
\end{proof}
\subsection{Proof of Proposition \ref{label:normapprox}}
\begin{proof}

We first show that for any $Q \in \cP(\Omega_{-T,T})$
\be\label{eq:projectcost}
C(Q,P^{\mu}_{-T}) \geq C(q,p)
\ee
where $q=(\omega_{-T},\omega_T)\# Q$ and $p= (\omega_{-T},\omega_T)\#P^{\mu}_{-T} $. Indeed we have, using conditional Jensen's inequality:
\beas
E_{P^{\mu}_{-T}}\Big(c\Big(\frac{dQ}{dP^{\mu}_{-T}}\Big) \Big) 	&=& E_{p}\Big(E_{P^{\mu}_{-T}}\Big[c(\frac{dQ}{d P^{\mu}_{-T}})\Big|\, \omega_{-T},\omega_T \Big]  \Big)			\\
								&\geq & E_{p}\Big(c\Big(E_{P^{\mu}_{-T}}\Big[\frac{dQ}{dP^{\mu}_{-T}}\Big| \, \omega_{-T},\omega_T \Big]\Big)  \Big)	\\
								&=&	E_{p}\Big(c\Big( \frac{d q}{dp} \Big)  \Big) =C(q,p)	
\eeas
Next, we shall show that:
\be\label{eq:projectreciproc}
Q \ll P^{\mu}_{-T}, Q \in \mathfrak{R}_T(U) \Rightarrow C(Q,P^{\mu}_{-T}) = C(q,p).
\ee
Indeed, under the above assumption $\frac{d Q}{d P^{\mu}_{-T}}$ is $(\omega_{-T},\omega_T)$-measurable, see e.g. \cite[Thm 2.13]{LRZ}. Therefore:
\bes
E_{P^{\mu}_{-T}}\Big(c\Big(\frac{dQ}{dP^{\mu}_{-T}}\Big) \Big) 	= E_{p}\Big(c\Big(E_{P^{\mu}_{-T}}\Big[\frac{dQ}{dP^{\mu}_{-T}}\Big| \omega_{-T},\omega_T \Big] \Big) \Big)	= E_{p}\Big(c\Big(\frac{dq}{dp} \Big) \Big) = c\big(q,p\big)
\ees
Let now $\hat{Q}$ be the optimal solution of \eqref{approxprob}: such solution has to be unique as $c$ is strictly convex. Consider now the measure $P^{\hat{\pi}}_{-T,T}$ (see \eqref{eq:PMU} for the definition ) where $\hat{\pi}= (\omega_{-T},\omega_T) \#\hat{Q}$. 
By construction $P^{\hat{\pi}}_{-T,T} \in \mathfrak{R}_T(U)$, $P^{\hat{\pi}}_{-T,T} \ll P^{\mu}_{-T}$ and $(\omega_{-T},\omega_T) \# \hat{Q} =(\omega_{-T},\omega_T)\# P^{\hat{\pi}}_{-T,T} = \hat{\pi} $. Using this, equation \eqref{eq:projectcost} and then \eqref{eq:projectreciproc}:
\bes
c(\hat{Q},P^{\mu}_{-T}) \geq c(\hat{\pi},p) = c(P^{\hat{\pi}}_{-T,T},P^{\mu}_{-T})
\ees
By optimality of $\hat{Q}$ and uniqueness,  it must be that $\hat{Q}=P^{\hat{\pi}}{-T,T}$ and hence $\hat{Q} \in \mathfrak{R}_T(U)$.
\end{proof}

\bibliographystyle{plain}
\bibliography{../command&Ref/Ref}

\end{document}